\documentclass[11pt]{article}

\usepackage{amssymb}
\usepackage{mathtools}
\usepackage[numbers]{natbib}
\usepackage[margin=1in]{geometry}

\usepackage{titlesec}
\titleformat{\section}
{\centering\large\bfseries}{\thesection.}{.5em}{}
\titleformat{\subsection}
{\large}{\thesubsection.}{.5em}{\bfseries}

\usepackage{amsthm}
\newtheorem{thm}{Theorem}[section]
\newtheorem{lem}[thm]{Lemma}

\newtheorem{cor}[thm]{Corollary}
\newtheorem{conj}[thm]{Conjecture}
\theoremstyle{definition}

\theoremstyle{remark}

\usepackage{stucky}

\usepackage{chngcntr}
\counterwithin*{equation}{section}

\begin{document}
\bibliographystyle{abbrv}

\begin{center}
	{\large\bf CERTAIN FLOOR FUNCTION SUMS OVER POWERS}\\[1.5em]
	{\scshape Joshua Stucky}
\end{center}
\vspace{.5em}

\begin{abstract}
We study the sums
\[
S_f(x) = \sum_{n\leq x} f\pth{\floor{\frac{x}{n}}}
\]
when $f$ is supported on $r$th powers with $r\geq 2$. This restriction allows us to give nontrivial estimates for one of the error terms in the asymptotic expansion of $S_f(x)$. We also state several conjectures related to our results.
\end{abstract}

\section{Introduction and Statement of Results}

Since their recent introduction by Bordell\'{e}s et.\ al.\ \cite{BordellesEtAl2019}, there has been a fair amount of research into the sums
\[
S_f(x) = \sum_{n\leq x} f\pth{\floor{\frac{x}{n}}},
\]
where $f$ is an arithmetic function and $\floor{\cdot}$ denotes the integer part function. Under mild assumptions on the growth of $f$, one can establish decent asymptotic estimates for such sums. For instance, Wu \cite{Wu2019-1} and Zhai \cite{Zhai2019} showed independently that if $f(n) \ll n^\alpha$ for some $0\leq \alpha < 1$, then
\begin{equation}\label{eq:WuZhai}
S_f(x) = C_f x + O\pth{x^{\frac{1+\alpha}{2}}},
\end{equation}
where here and throughout we denote
\begin{equation}\label{eq:CfDef}
C_f = \sum_{n=1}^\infty \frac{f(n)}{n(n+1)}.
\end{equation}
This estimate has been improved by a number of authors for functions satisfying various additional properties. For instance, if one has a mean square estimate
\[
\sum_{n\leq x} \abs{f(n)}^2 \ll x^\alpha
\]
for some $\alpha\in[1,2)$, then Wu and Zhao \cite{WuZhao2022} have shown that
\[
C_f x + O\pth{x^\frac{2+3\alpha}{8}}.
\]

In a different direction, one can use exponential sum estimates to improve \eqref{eq:WuZhai} when $f$ has a convolution structure, say $f=1*g$, or an analogous structure such as that given by Vaughan's identity in the case $f = \Lambda$ (see, e.g. \cite{LiuWuYang2021General},\cite{LiuWuYang2021PNT},\cite{Stucky2022},\cite{Zhang2021}). 

In this article, we introduce another property that $f$ may possess that allows us to improve \eqref{eq:WuZhai}. Specifically, we consider functions of the form
\begin{equation}\label{eq:fOnPowers}
f(n) = \begin{cases}
h(d) & \text{if $n=d^r$},\\
0 & \text{otherwise}.
\end{cases}
\end{equation}
The reason we consider such functions is that when $r\geq 2$, we can give a nontrivial estimate for one of the error terms arising in the asymptotic formula for $S_f$, namely
\begin{equation}\label{eq:ColefDef}
S_f^\dagger(x;A) = \sum_{n < A} f\pth{\floor{\frac{x}{n}}}.
\end{equation}
In all previous works, this quantity is estimated trivially using a point-wise upper bound for $f$. For functions supported on $r$th powers, we are able to improve this estimate by understanding the spacing between the integers $n$ for which there exists an integer $d$ satisfying
\[
\floor{\frac{x}{n}} = d^r.
\]
We do so by following the method of Roth and Halberstam (see, e.g., \cite{FGT2015} for a nice survey of these ideas). As an example of our results, if $h(n) \ll n^\alpha$ for some $\alpha \in [0,\frac{1}{2r-1})$, then for $A \ll x^{\frac{1}{r+1}}$,
\begin{equation}\label{eq:ExampleEstimate}
S_f^\dagger\pth{x;A}  \ll \fracp{x}{A}^{\frac{\alpha}{r}} A^{\frac{1}{r(2r-1)}} x^{\frac{r-1}{r(2r-1)}},
\end{equation}
which is better than the trivial estimate
\begin{equation}\label{eq:SfDaggerTrivial}
S_f^\dagger\pth{x;A}  \ll A \fracp{x}{A}^{\frac{\alpha}{r}}
\end{equation}
so long as $A \geq x^{\frac{1}{2r+1}}$. Using Lemma \ref{lem:Spacing}, we will prove the following generalization of \eqref{eq:WuZhai}.

\begin{thm}\label{thm:GeneralWuZhai}
Let $r\geq 1$ be an integer and let
\begin{equation}\label{eq:etaDef}
\eta = \begin{cases}
	1 & \text{if $r$ is odd}, \\
	0 & \text{if $r$ is even}.
\end{cases}
\end{equation}
If $f$ is given by {\normalfont\eqref{eq:fOnPowers}} with $h$ satisfying $\abs{h(d)} \ll d^\alpha$ for some $\alpha\in[0,\frac{r+2-\eta}{2r-2+2\eta})$, then
\begin{equation}\label{eq:GeneralWuZhai}
S_f(x) = C_f x + O\pth{x^{\frac{2(1+\alpha)}{3r+\eta}}},
\end{equation}
where $C_f$ is given by {\normalfont \eqref{eq:CfDef}}. 
\end{thm}

We note that the trivial estimate \eqref{eq:SfDaggerTrivial} leads to
\[
S_f(x) = C_f x + O\pth{x^{\frac{1+\alpha}{r+1}}},
\]
which is weaker than \eqref{eq:GeneralWuZhai} when $r\geq 3$. The form of Theorem \ref{thm:GeneralWuZhai} is someone unnatural, owing primarily to the restriction on the range of $A$ in Lemma \ref{lem:Spacing}. It seems reasonable that the following simpler and stronger result should hold for a more restricted range of $\alpha$.

\begin{conj}\label{conj:Trivial}
Let $r\geq 2$ be an integer. If $f$ is given by {\normalfont(\ref{eq:fOnPowers})} with $h$ satisfying $\abs{h(d)} \ll d^\alpha$ for some $\alpha\in[0,\frac{1}{2r-1}]$, then
\[
S_f(x) = C_f x + O\pth{x^{\frac{1+\alpha}{2r}}},
\]
where $C_f$ is given by {\normalfont(\ref{eq:CfDef})}. 
\end{conj}

It is possible that Conjecture \ref{conj:Trivial} holds for a larger range of $\alpha$. The reason to consider only the range given in Conjecture \ref{conj:Trivial} is that if we take $A$ maximally in \eqref{eq:ExampleEstimate}, then
\[
S_f^\dagger\pth{x;x^{\frac{1}{r+1}}} \ll x^{\frac{r}{(r+1)(2r-1)} + \frac{\alpha}{r+1}} \ll x^{\frac{1+\alpha}{2r}} 
\]
precisely when $\alpha \leq \frac{1}{2r-1}$. Conjecture \ref{conj:Trivial} then follows from the estimate
\[
\sumabs{\sum_{n\leq x^{\frac{1}{r+1}}} \psi\fracp{x}{n^r+\delta}} \ll x^{\frac{r}{(r+1)(2r-1)}},
\]
where $\delta = 0$ or $1$ and $\psi(t) = t- \floor{t} - \frac{1}{2}$. Using estimates for exponential sums, we are able to prove this in the case $r=2$, as well as an estimate for general $r$ assuming the existence of certain exponent pairs.

\begin{thm}\label{thm:ConjProof}
Let $r\geq 2$	and let $f$ be given by {\normalfont(\ref{eq:fOnPowers})} with $h$ satisfying $\abs{h(d)} \ll d^\alpha$ for some $\alpha\in[0,\frac{1}{2r-1}]$. Suppose there exists an exponent pair $(k,\ell)$ such that $\ell=rk$. Then
\[
S_f(x) = C_f x + O\pth{x^{\frac{\alpha}{r+1}+\frac{k}{k+1}}\log x+ + x^{\frac{1+\alpha}{2r}}},
\]
where the factor $\log x$ is present only when $\alpha = 0$.
\end{thm}

This implies the following improvements to Theorem \ref{thm:GeneralWuZhai} in the cases $r=2,3,4$.

\begin{cor}\label{cor:ConjProof}${}$
\begin{enumerate}[label={\normalfont \arabic*.}]
\item Let $r=2$ and let $f$ be given by {\normalfont(\ref{eq:fOnPowers})} with $h$ satisfying $\abs{h(d)} \ll d^\alpha$ for some $\alpha\in[0,\frac{1}{3}]$. Then
\[
S_f(x) = C_f x + O\pth{x^{\frac{\alpha}{3}+\frac{2}{9}}\log x + x^{\frac{1+\alpha}{4}}},
\]

\item Let $r=3$ and let $f$ be given by {\normalfont(\ref{eq:fOnPowers})} with $h$ satisfying $\abs{h(d)} \ll d^\alpha$ for some $\alpha\in[0,\frac{1}{5}]$. Then
\[
S_f(x) = C_f x + O\pth{x^{\frac{\alpha}{4}+\frac{11}{64}}\log x}.
\]
\item Let $r=4$ and let $f$ be given by {\normalfont(\ref{eq:fOnPowers})} with $h$ satisfying $\abs{h(d)} \ll d^\alpha$ for some $\alpha\in[0,\frac{1}{7}]$. Then
\[
S_f(x) = C_f x + O\pth{x^{\frac{\alpha}{5}+\frac{1}{7}}\log x}.
\]
\end{enumerate}
In each case, $C_f$ is given by {\normalfont\eqref{eq:CfDef}} and the factor $\log x$ is present only when $\alpha = 0$.  In particular, Conjecture \ref{conj:Trivial} holds in the case $r=2$.
\end{cor}

In light of Theorem \ref{thm:ConjProof}, we also make the following general conjecture regarding exponent pairs, which may be of independent interest.

\begin{conj}
For any real $\ep > 0$ and $r\geq 1$, there exists an exponent pair $(k,\ell)$ such that
\[
\abs{\frac{\ell}{k} - r} < \ep.
\]
\end{conj}

\noindent\textbf{Notation.} The Vinogradov-Landau symbols $O,o,\ll,\gg$ have their usual meanings, and we always allow the constants implied by these symbols to depend on the parameter $r$ and the function $f$. We use $n\sim N$ to denote the condition $N < n \leq 2N$ and use $n\asymp N$ to denote the more general condition $c_1N \leq n \leq c_2 N$ for some positive constants $c_1 < c_2$. A sum $\sumd$ denotes a dyadic sum over powers $2^k$. Throughout, $\eta$ is given by (\ref{eq:etaDef}), $e(x) = e^{2\pi i x}$, and $\psi(x) = x - \floor{x} - \frac{1}{2}$.

\section{A Spacing Lemma}\label{sec:Spacing}

In this section, we prove Lemma \ref{lem:Spacing}, which is the main input needed for Theorem \ref{thm:GeneralWuZhai}. To do so, we require the following lemma concerning certain polynomial identities. It is a special case of Lemma 2 of Huxley and Nair \cite{HuxleyNair80}.

\begin{lem}\label{lem:HuxleyNair}
Let $r\geq 2$ be a positive integer. For each integer $l$ with $1\leq l\leq r$, there exist polynomials $P_l$ and $Q_l$ such that, as $x\to 0$,
\begin{enumerate}[label=\normalfont{(\arabic*)}]
\item $\abs{P_l(x) (1-x)^r - Q_l(x)} \ll \abs{x}^{2l-1}$;
\item $P_l$ and $Q_l$ have degree $l-1$;
\item the coefficients of $P_l$ and $Q_l$ are nonzero integers.
\end{enumerate}
\end{lem}

\begin{lem}\label{lem:Spacing}
Let $r\geq 1$ and let $f$ be given by {\normalfont(\ref{eq:fOnPowers})}. Suppose $h(d) \leq C d^\alpha$ for some $\alpha,C \geq 0$. Then for each integer $l$ with $1\leq l\leq r$, there exists a constant $C_r$ depending only on $r$ such that for $A\leq C_r x^{\frac{r+1-l}{2r+1-l}}$, we have
\begin{equation}\label{eq:SpacingEstimate}
\sum_{n < A} f\pth{\floor{\frac{x}{n}}} \ll x^{\frac{l-1}{r(2l-1)}+\frac{\alpha}{r}} \pth{A^{\frac{r+1-l}{r(2l-1)} - \frac{\alpha}{r}} + 1}.
\end{equation}
\end{lem}

\begin{proof}
As stated in the introduction, we follow the method described in Section 1 of \cite{FGT2015}. Trivially
\begin{equation}\label{eq:TMDecomp}
\sum_{n < A} f\pth{\floor{\frac{x}{n}}} \ll \sumd_{\fracp{x}{A}^{\frac{1}{r}} \ll D \ll x^{1/r}}  D^{\alpha} T(D),
\end{equation}
where
\[
T(D) = \abs{\set{d\sim D: d^r = \floor{\frac{x}{n}}\ \text{for some $n$}}}.
\]
Suppose $d,d-a\in T(D)$ so that
\begin{equation}\label{eq:mVals}
d^r = \floor{\frac{x}{n_1}}, \qquad (d-a)^r = \floor{\frac{x}{n_2}}.
\end{equation}
Removing the floor brackets and rearranging, we have
\begin{equation}\label{eq:nsNoFloor}
n_1 = \frac{x}{d^r} + O\fracp{x}{D^{2r}}, \qquad n_2 = \frac{x}{(d-a)^r} + O\fracp{x}{D^{2r}}.
\end{equation}
Using Lemma \ref{lem:HuxleyNair}, there exist polynomials $P_l,Q_l$ such that for $a = o(D)$, we have
\[
\abs{P_l\fracp{a}{d}\pth{1-\frac{a}{d}}^r - Q_l\fracp{a}{d}} \ll \fracp{a}{D}^{2l-1}.
\]
Multiplying through by $d^{r+l-1}$, we have
\begin{equation}\label{eq:PolyBound}
\abs{P(a,d)\pth{d-a}^r - Q(a,d)d^r} \ll a^{2l-1} D^{r-l}
\end{equation}
for some homogeneous polynomials $P_0(a,d), Q_0(a,d)$ of degree $l-1$ with integer coefficients. Consider now the modified difference
\[
\cold = P_0(a,d) n_1 - Q_0(a,d) n_2.
\]
By (\ref{eq:nsNoFloor}) and (\ref{eq:PolyBound}),
\[
\abs{\cold} \ll \frac{xa^{2l-1}}{D^{l+r}} + \frac{x}{D^{2r+1-l}}.
\]
The inequality $A \leq C_r x^{\frac{r+1-l}{2r+1-l}}$ implies
\begin{equation}\label{eq:Dreq}
D \gg \fracp{x}{A}^{\frac{1}{r}} \geq C_r' x^{\frac{1}{2r+1-l}}
\end{equation}
for some constant $C_r'$. Taking $C_r$ sufficiently small, we see that there is a constant $C_r''$ such that
\begin{equation}\label{eq:ColdAbs}
\abs{\cold} \leq C_r''\frac{xa^{2l-1}}{D^{l+r}} + \frac{1}{2}.
\end{equation}
Since $\cold$ is an integer, it follows that if $a \leq C_r''' D^{\frac{l+r}{2l-1}} x^{-\frac{1}{2l-1}} = L$, say, for some sufficiently small constant $C_r'''$, then $\abs{\cold} < 1$, so $\cold = 0$.

Suppose now that $I$ is a subinterval of $(D,2D]$ of length at most $L$. Using the above analysis, we will show that $I$ contains at most $2l$ elements of $T(D)$. Suppose that $d,d-a,d-a-b \in T(D)$, so
\[
d^r = \floor{\frac{x}{n_1}}, \qquad (d-a)^r = \floor{\frac{x}{n_2}}, \qquad (d-a-b)^r = \floor{\frac{x}{n_3}}.
\]
Then since $\abs{I}\leq L$, it follows that
\[
\begin{aligned}
P_0(a,d) n_2 - Q_0(a,d) n_1 &= 0 \\
P_0(a+b,d) n_3 - Q_0(a+b,d) n_1 &= 0\\
P_0(b,d+a) n_3 - Q_0(b,d+a) n_2 &= 0
\end{aligned}
\]
We multiply the second of these by $P_0(b,d+a)$ and the third by $P_0(a+b,d)$, then subtract to see that
\begin{equation}\label{eq:polyEquals0}
P_0(a+b,d)Q_0(b,d+a) n_2 - P_0(b,d+a)Q_0(a+b,d) n_1 = 0.
\end{equation}
We view $d$ and $a$ as fixed and view the left hand side as a polynomial in $b$. Note that, by construction, $P_0(a+b,d)Q_0(b,d+a)$ and $P_0(b,d+a)Q_0(a+b,d)$ have the same leading coefficient as polynomials in $b$, namely $v_P v_Q$, where $v_P$ and $v_Q$ are the leading coefficients of $P_l$ and $Q_l$, respectively. Clearly we cannot have $n_1=n_2$, and therefore the left side of (\ref{eq:polyEquals0}) is a polynomial of degree $2l-2$. Thus (\ref{eq:polyEquals0}) has at most $2l-2$ solutions in $b$, and it follows that any interval $I$ of length at most $L$ contains at most $2l$ elements of $T(D)$. Dividing the interval $(D,2D]$ into $O(DL^{-1})$ intervals of length at most $L$, we deduce that
\[
T(D) \ll \frac{D}{L} +1 \ll \fracp{x}{D^{r-l+1}}^{\frac{1}{2l-1}}.
\]
We conclude the proof of Lemma \ref{lem:Spacing} by substituting this into \eqref{eq:TMDecomp} and executing the dyadic sum in $D$.

\end{proof}

\section{Proofs of Theorems \ref{thm:GeneralWuZhai} and \ref{thm:ConjProof}} 

\subsection{Main and Error Terms}
Similar to previous works, we decompose $S_f(x)$ via
\[
S_f(x) = S_f^\dagger(x;B)  + S_f^\flat(x;A;B) + S_f^\sharp(x;A),
\]
where
\[
\begin{aligned}
S_f^\dagger(x;B) &=  \sum_{n \leq B} f\pth{\floor{\frac{x}{n}}}, \\
S_f^\flat(x;A;B) &=  \sum_{\fracp{x}{A}^{\frac{1}{r}} < d < \fracp{x}{B}^{\frac{1}{r}}} h(d) \pth{\floor{\frac{x}{d^r}} - \floor{\frac{x}{d^r+1}}}, \\
S_f^\sharp(x;A) &=  \sum_{d\leq \fracp{x}{A}^{\frac{1}{r}}} h(d) \pth{\floor{\frac{x}{d^r}} - \floor{\frac{x}{d^r+1}}},
\end{aligned}
\]
and $1 \leq B \leq A \leq \sqrt{x}$ are parameters to be chosen. We treat these sums as follows. 

For $S_f^\dagger(x;B)$, we apply Lemma \ref{lem:Spacing}. In the range
\begin{equation}\label{eq:ParamRange}
1\leq l\leq r,\qquad \alpha < \frac{r+1-l}{2l-1},\qquad B\leq C_r x^{\frac{r+1-l}{2r+1-l}},
\end{equation}
we have
\begin{equation}\label{eq:SpacingBound}
S_f^\dagger(x;B) \ll \fracp{x}{B}^{\frac{\alpha}{r}} x^{\frac{l-1}{r(2l-1)}} B^{\frac{r+1-l}{r(2l-1)}}.
\end{equation}

For $S_f^\flat(x;A;B)$, we have
\[
\begin{aligned}
\abs{S_f^\flat(x;A;B)} &\leq \sum_{\fracp{x}{A}^{\frac{1}{r}} < d < \fracp{x}{B}^{\frac{1}{r}}} \abs{h(d)} \pth{\floor{\frac{x}{d^r}} - \floor{\frac{x}{d^r+1}}} \\
&\ll \sum_{\fracp{x}{A}^{\frac{1}{r}} < d < \fracp{x}{B}^{\frac{1}{r}}} d^\alpha \pth{\floor{\frac{x}{d^r}} - \floor{\frac{x}{d^r+1}}} \\
&\leq x \sum_{\fracp{x}{A}^{\frac{1}{r}} < d < \fracp{x}{B}^{\frac{1}{r}}} \frac{d^\alpha}{d^r(d^r+1)} + \abs{\cole_0^\flat(x;A,B)} + \abs{\cole_1^\flat(x;A,B)},
\end{aligned}
\]
where
\begin{equation}\label{eq:ColeFlatDef}
\cole_\delta^\flat(x;A,B)  = \sum_{\fracp{x}{A}^{\frac{1}{r}} < d < \fracp{x}{B}^{\frac{1}{r}}} d^\alpha \psi\fracp{x}{d^r+\delta}.
\end{equation}
Since $r \geq 2$, we have $\alpha < 2r-1$ by \eqref{eq:ParamRange}. Thus
\begin{equation}\label{eq:FracTailBound}
S_f^\flat(x;A;B) \ll \frac{A^2}{x}\fracp{x}{A}^{\frac{1+\alpha}{r}} + \abs{\cole_0^\flat(x;A,B)} + \abs{\cole_1^\flat(x;A,B)}.
\end{equation}

Finally, for $S_f^\sharp(x;A)$, we remove the floor brackets to get
\[
S_f^\sharp(x;A) = x \sum_{d\leq \fracp{x}{A}^{\frac{1}{r}}} \frac{h(d)}{d^r(d^r+1)} + \sum_{d\leq \fracp{x}{A}^{\frac{1}{r}}} h(d) \pth{\psi\pth{\frac{x}{d^r+1}} - \psi\pth{\frac{x}{d^r}}}.
\]
The first sum is
\[
x \sum_{d\leq \fracp{x}{A}^{\frac{1}{r}}} \frac{h(d)}{d^r(d^r+1)} = C_f x + O\pth{\frac{A^2}{x}\fracp{x}{A}^{\frac{1+\alpha}{r}}},
\]
and thus
\begin{equation}\label{eq:Main+Error}
S_f^\sharp(x;A) = C_f x + O\sumpth{ \frac{A^2}{x}\fracp{x}{A}^{\frac{1+\alpha}{r}} + \abs{\cole_0^\sharp(x;A)} + \abs{\cole_1^\sharp(x;A)}},
\end{equation}
where
\begin{equation}\label{eq:ColeSharpDef}
\cole_{\delta}^\sharp(x;A) =\sum_{d\leq \fracp{x}{A}^{\frac{1}{r}}} h(d)\psi\pth{\frac{x}{d^r+\delta}}.
\end{equation}
Combining \eqref{eq:SpacingBound}, \eqref{eq:FracTailBound}, and \eqref{eq:Main+Error}, we have, in the range \eqref{eq:ParamRange},
\begin{equation}\label{eq:FourErrorTerms}
S_f(x) = C_f x + O\pth{\frac{A^2}{x}\fracp{x}{A}^{\frac{1+\alpha}{r}} + \fracp{x}{B}^{\frac{\alpha}{r}} x^{\frac{l-1}{r(2l-1)}} B^{\frac{r+1-l}{r(2l-1)}} + \cole(x;A,B)},
\end{equation}
where 
\begin{equation}\label{eq:ColeDef}
\cole(x;A,B) = \abs{\cole_0^\flat(x;A,B)} + \abs{\cole_1^\flat(x;A,B)}+ \abs{\cole_0^\sharp(x;A)} + \abs{\cole_1^\sharp(x;A)}
\end{equation} 
and $\cole_\delta^\flat$ and $\cole_\delta^\sharp$ are given by \eqref{eq:ColeFlatDef} and \eqref{eq:ColeSharpDef}, respectively. We note the trivial bound
\begin{equation}\label{eq:coleTrivial}
\cole(x;A,B) \ll \fracp{x}{B}^{\frac{1+\alpha}{r}}.
\end{equation}

\subsection{Proof of Theorem \ref{thm:GeneralWuZhai}} \label{sec:WuZhai}

Using the trivial estimate \eqref{eq:coleTrivial} with $B = A$ and recalling that we will choose $A \leq \sqrt{x}$, we have, in the range \eqref{eq:ParamRange}
\[
S_f(x) = C_f x + O\pth{\fracp{x}{A}^{\frac{\alpha}{r}} \pth{\fracp{x}{A}^{\frac{1}{r}} + x^{\frac{l-1}{r(2l-1)}}A^{\frac{r+1-l}{r(2l-1)}}} }.
\]
The optimal choice $A = x^{\frac{l}{r+l}}$ gives the estimate
\[
S_f(x) = C_f x + O\pth{x^{\frac{1+\alpha}{l+r}}}.
\]
However, we must choose $A$ subject to (\ref{eq:ParamRange}), and this requires
\[
l \leq \frac{r+1}{2}.
\]
We conclude the proof of Theorem \ref{thm:GeneralWuZhai} by choosing $l = \frac{r+\eta}{2}$.

\subsection{Proof of Theorem \ref{thm:ConjProof} }

Recall that
\[
\cole_\delta^\flat(x;A,B)  = \sum_{\fracp{x}{A}^{\frac{1}{r}} < d < \fracp{x}{B}^{\frac{1}{r}}} d^\alpha \psi\fracp{x}{d^r+\delta}.
\]
Subdividing into dyadic intervals and applying partial summation, we have
\[
\cole_\delta^\flat(x;A,B) \ll \sumd_{\fracp{x}{A}^{\frac{1}{r}} \ll D \ll \fracp{x}{B}^{\frac{1}{r}}} D^\alpha \max_{D \leq D' \leq 2D} \abs{\sum_{D < d \leq D'} \psi\fracp{x}{d^r+\delta}}.
\]
We estimate the sum over $d$ via the following lemma, which is a special case of Lemma 4.3 of \cite{GrahamKolesnik}.

\begin{lem}[Graham and Kolesnik]\label{lem:GK}
Let $g(n) \asymp y n^{-r}$ for some $r > 0$ and $g\in C^1([N,2N])$ with $N \geq 1$. Then for any exponent pair $(k,\ell)$ and any $N'\in[N,2N]$, we have
\[
\sum_{N < n \leq N'} \psi(g(n)) \ll y^{\frac{k}{k+1}} N^{\frac{\ell-rk}{k+1}} + \frac{N^{r+1}}{y}.
\]
\end{lem}

Applying this to $\cole_\delta^\flat$ yields
\begin{equation}\label{eq:GKResult}
\cole_\delta^\flat(x;A,B) \ll \sumd_{\fracp{x}{A}^{\frac{1}{r}} \ll D \ll \fracp{x}{B}^{\frac{1}{r}}} D^\alpha \pth{x^{\frac{k}{k+1}} D^{\frac{\ell-rk}{k+1}} + \frac{D^{r+1}}{x}}.
\end{equation}
By hypothesis, there exists an exponent pair $(k,\ell)$ with $\ell=rk$, and so
\[
\cole_\delta^\flat(x;A,B) \ll \sumd_{\fracp{x}{A}^{\frac{1}{r}} \ll D \ll \fracp{x}{B}^{\frac{1}{r}}} D^\alpha \pth{x^{\frac{k}{k+1}} + \frac{D^{r+1}}{x}} \ll \fracp{x}{B}^\frac{\alpha}{r} \pth{x^{\frac{k}{k+1}}\log x + \frac{x^\frac{1}{r}}{B^{1+\frac{1}{r}}}},
\]
where the factor $\log x$ is present only when $\alpha = 0$. Estimating $\cole_\delta^\sharp(x;A)$ trivially, we find that
\[
\cole(x;A,b) \ll \fracp{x}{B}^\frac{\alpha}{r} \pth{x^{\frac{k}{k+1}}\log x + \frac{x^\frac{1}{r}}{B^{1+\frac{1}{r}}}} + \fracp{x}{A}^{\frac{1+\alpha}{r}}.
\]
We conclude the proof Theorem \ref{thm:ConjProof} by inserting the above estimate into \eqref{eq:FourErrorTerms} and specifying $A = \sqrt{x}$, $B = C_rx^{\frac{1}{r+1}}$, and $l=r$. Corollary \ref{cor:ConjProof} then follows from the exponent pairs
\[
\pth{\frac{2}{7},\frac{4}{7}} = BA^2\pth{\frac{1}{2},\frac{1}{2}}, \qquad \pth{\frac{11}{53},\frac{33}{53}} = BABA^2BA^2\pth{\frac{1}{2},\frac{1}{2}}, \qquad \pth{\frac{1}{6},\frac{2}{3}} = A\pth{\frac{1}{2},\frac{1}{2}}
\] 
for $r=2,3,4$, respectively. Here $A$ and $B$ denote the usual $A$ and $B$ processes of obtaining new exponent pairs.

\bibliography{references}

\begin{thebibliography}{10}

\bibitem{BordellesEtAl2019}
O.~Bordell\`es, L.~Dai, R.~Heyman, H.~Pan, and I.~E. Shparlinski.
\newblock On a sum involving the {E}uler function.
\newblock {\em J. Number Theory}, 202:278--297, 2019.

\bibitem{FGT2015}
M.~Filaseta, S.~Graham, and O.~Trifonov.
\newblock Starting with gaps between {$k$}-free numbers.
\newblock {\em Int. J. Number Theory}, 11(5):1411--1435, 2015.

\bibitem{GrahamKolesnik}
S.~W. Graham and G.~Kolesnik.
\newblock {\em van der {C}orput's method of exponential sums}, volume 126 of
  {\em London Mathematical Society Lecture Note Series}.
\newblock Cambridge University Press, Cambridge, 1991.

\bibitem{HuxleyNair80}
M.~N. Huxley and M.~Nair.
\newblock Power free values of polynomials. {III}.
\newblock {\em Proc. London Math. Soc. (3)}, 41(1):66--82, 1980.

\bibitem{LiuWuYang2021General}
K.~Liu, J.~Wu, and Z.~Yang.
\newblock On some sums involving the integral part function.
\newblock 2021.
\newblock arXiv 2109.01382.

\bibitem{LiuWuYang2021PNT}
K.~Liu, J.~Wu, and Z.~Yang.
\newblock A variant of the prime number theorem.
\newblock {\em Indag. Math. (N.S.)}, 33(2):388--396, 2022.

\bibitem{Stucky2022}
J.~Stucky.
\newblock The fractional sum of small arithmetic functions.
\newblock {\em J. Number Theory}, 238:731--739, 2022.

\bibitem{Wu2019-1}
J.~Wu.
\newblock On a sum involving the {E}uler totient function.
\newblock {\em Indag. Math. (N.S.)}, 30(4):536--541, 2019.

\bibitem{Zhai2019}
W.~Zhai.
\newblock On a sum involving the {E}uler function.
\newblock {\em J. Number Theory}, 211:199--219, 2020.

\bibitem{Zhang2021}
W.~Zhang.
\newblock On general sums involving the floor function with applications to
  $k$-free numbers, 2021.
\newblock arXiv 2112.06156.

\bibitem{WuZhao2022}
F.~Zhao and J.~Wu.
\newblock Note on a paper by {B}ordell\`es, {D}ai, {H}eyman, {P}an and
  {S}hparlinski, 2.
\newblock {\em Acta Arith.}, 202(2):185--194, 2022.

\end{thebibliography}

\end{document}